\documentclass[11pt,twoside]{article}
\usepackage{latexsym}
\usepackage{amssymb,amsbsy,amsmath,amsfonts,amssymb,amscd,amsthm}
\usepackage{graphicx}
\usepackage{hyperref}
\newcommand{\commentout}[1]{}

\newcommand {\wt}[1] {{\widetilde #1}}
\newcommand {\wh}[1] {{\widehat #1}}

\newcommand{\E}{\mathbb{E}}
\newcommand{\R}{\mathbb{R}}

\newcommand {\al} {\alpha}
\newcommand {\e}  {\varepsilon}

\newcommand {\lb} {\lambda}

\newcommand {\dv}  { {\rm div} }
\newcommand {\f}   {\frac}
\newcommand {\p}   {\partial}
\newcommand{\beq}{\begin{equation}}
\newcommand{\beqa}{\begin{eqnarray}}
\newcommand{\bea} {\begin{array}{ll}}
\newcommand{\beqan}{\begin{eqnarray*}}
\newcommand{\eeq}{\end{equation}}
\newcommand{\eeqa}{\end{eqnarray}}
\newcommand{\eeqan}{\end{eqnarray*}}
\newcommand{\eea} {\end{array}}
\newtheorem{theorem}{Theorem}[section]
\newtheorem{lemma}[theorem]{Lemma}

\newcommand{\cqfd}{{ \hfill
                       {\unskip\kern 6pt\penalty 500
                       \raise -2pt\hbox{\vrule\vbox to 6pt{\hrule width 6pt
                       \vfill\hrule}\vrule} \par}   }}
\title{\Large \bf Stochastic averaging lemmas for kinetic equations}

\author{ Pierre-Louis Lions$^{1}$, Beno\^ \i t Perthame$^{2}$ and Panagiotis E. Souganidis$^{3,4}$}
\date{\today}

\begin{document}
\maketitle

\pagenumbering{arabic}

\begin{abstract}
We develop a class of averaging lemmas for stochastic kinetic equations. The velocity is multiplied by a white noise which produces a remarkable change in time scale. 

Compared to the deterministic case and as far as we work in $L^2$, the nature of regularity on averages is not changed in this stochastic kinetic equation and stays in the range of fractional Sobolev spaces  at the price of an additional expectation. However all the exponents are changed; either time decay rates are slower (when the right hand side belongs to $L^2$), or regularity is better when the right hand side contains derivatives. These changes originate from a different space/time scaling in the deterministic and stochastic cases.

Our motivation comes from scalar conservation laws with stochastic fluxes where the structure under consideration arises naturally through the kinetic formulation of scalar conservation laws.
\end{abstract}
\pagestyle{plain} \vspace*{1.0cm}
\noindent {\bf Key words.}  Stochastic kinetic equations; stochastic conservation laws; averaging lemmas; fractional Sobolev spaces. 
\\
\noindent {\bf AMS Class. Numbers.}  35L65; 35R60; 60H15; 82C40
\tableofcontents

\section{Introduction}
\label{sec:intro}

The kinetic formulation of the stochastic conservation laws, as developed in \cite{LPS_sscl}, motivates to study  the stochastic kinetic equation
\begin{equation}\label{skinetic}
\begin{cases}
\f{\p}{\p t} f(x,\xi,t)  + \dot B(t) \circ \xi . \nabla_x f = g(x, \xi,t )  \quad \text{ in } \quad \R^{2d} \times(0,\infty), 
\\[2mm]
f(t=0,x,\xi)=f^0(x,\xi).
\end{cases}
\end{equation}
where both $f^0$ and $g$ are deterministic. 
The notation for the flux means
$$
\dot B(t) \circ \xi . \nabla_x f = \dot  B(t)  \circ \sum_{i=1}^d  \xi_i \frac{\p f}{\p x_i}.
$$
That is we we use Stratonovich convention. Also we introduce a single brownian motion (for simplicity). At variance with the examples of Hamilton-Jacobi equations \cite{LShjscrasEDP,LShjs} or $x$-dependent fluxes in scalar conservation laws \cite{LPS_sscl}, this does not play an important role here. 
\\

Our results are mainly motivated by the case when the $B$ is a brownian motion. But for reference we will use also the standard deterministic case $B(t)=t$.  In this deterministic case, the averaging lemmas are now sophisticated \cite{DLM,Ger, PS, Bouchut_av, BP_bams}  and have proven to be useful for treating nonlinear kinetic equations as the Vlasov-Maxwell \cite{DiLi} or the Boltzmann equation \cite{DiLi_B,Vi}. 
\\

We understand this equation in the spirit of the works on Hamilton-Jacobi equations in the series of papers by Lions and Souganidis, see \cite{LShjs, LShjscrasEDP} and the references therein. That is, we define solutions as the limit of standard distributional solutions when $B(t)$ is regularized; in other words they correspond to a Stratonovich rather than an It{\=o}-Doeblin integral. This is a difference compared to stochastic  transport equation also treated by Flandoli \cite{Flandoli, Flandoli_ln} or with conservation laws with  stochastic semilinear terms treated by Debussche and Vovelle \cite{dv} or with random kinetic equations, with a stochastic semilinear term that are also treated by Debussche and Vovelle in \cite{dv2} where the diffusion limit is studied.   
\\

The question of the existence is not important here because of linearity and the method of characteristics gives the representation formula
$$
\frac{d}{dt} f\big(x + B(t) \circ \xi, \xi ,t \big) = g\big(x + B(t) \circ \xi, \xi ,t \big) .
$$
This defines our solutions (see \cite{LShjscrasEDP,LShjs}  for a motivation and \cite{LPS_sscl} for a more interesting nonlinear case). This is possible because we consider the Stratonovich convention; as usual it can be transformed in an  It{\=o} integral with an additional diffusion term, \cite{Flandoli, Flandoli_ln}.

For a given test function $\psi(\xi)$ with compact support, we define the random average
$$
\rho_\psi(x,t) =\int \psi(\xi)  f(x,\xi,t) d\xi.
$$
Our interest lies in the {\em velocity averaging lemmas} which are regularity statements on $\rho_\psi$; how to state them? how do the fractional exponents reflect the scales of the brownian motion? 

We answer these questions by a series of results covering a part of what is known for averaging lemmas as stated in long series of papers, from the initial remark on the phenomena \cite{GPS,GLPS} for $p=2$, to the optimal cases of full space derivatives in the right hand side \cite{Ger,PS}, including the case of $L^p$ spaces for $p\neq 2$ \cite{DLM}. 

We work in $L^2$ and we measure time decay thanks to an additional time damping. This is a way to better visualize the different scales between space and time that appear in the stochastic case compared to the deterministic case.

\section{Space regularity, global in time}
\label{sec:space}

When considering space regularity only, we can obtain particularly simple results which are global in time and thus also express a time decay. They are useful to make a difference between deterministic and stochastic scales. 

\begin{theorem}[Stochastic averaging] Take $B$ a brownian motion and $\lb >0$. 
\\
For $g=0$
$$
\E \| e^{-\lb t} \rho_\psi\|^2_{L^2\big(R^+; \dot H^{1/2}(\R^d) \big)} \leq \f{C({\rm supp} \; \psi )}{\lb^{1/2}} \; \|Ê\psiÊf^0 \|^2_{L^2(\R^d\times \R^d)} ,
$$
$$
\E \| \rho_\psi\|^2_{L^2\big(R^+;  \dot H^{1/2}(\R^d) \big)} \leq C({\rm supp} \; \psi ) \; \|Ê\psi Êf^0 \|_{L^2(\R^d\times \R^d)} \|Ê\psi Êf \|_{L^2(  \R^+\times\R^d\times \R^d )} .
$$
For $f^0=0$
$$
\E \|e^{-\lb t} \rho_\psi\|^2_{L^2\big(R^+;  \dot H^{1/2}(\R^d) \big)} \leq  \f{C({\rm supp} \; \psi )}{\lb^{3/2}} \; \|Êe^{-\lb t} \psi g  \|_{L^2(\R^+ \times \R^d\times \R^d)} ,
$$
$$
\E  \| \rho_\psi\|^2_{L^2\big(R^+; \dot  H^{1/2}(\R^d) \big)} \leq  C({\rm supp} \; \psi ) \; \|Ê\psi g \|^{1/2}_{L^2(\R^+ \times \R^d\times \R^d)} \|Ê\psi f \|^{3/2}_{L^2(\R^+ \times \R^d\times \R^d)} .
$$
For $f^0=0$ and $g = \dv \; h$, we have 
$$
\E  \| \rho_\psi\|^2_{L^2\big(R^+;  \dot H^{1/3}(\R^d) \big)} \leq C(\psi ) \; \left[ \|Ê h\|^2_{\dot H_x^{-2/3}\cap L^2(\R^+ \times \R^d\times \R^d)} + \| \psi f \|^2_{L^2(\R^+ \times \R^d\times \R^d)}  \right] .
$$
\label{th:aver} 
\end{theorem}

Even though some exponents may seem similar, all the scales differ from the deterministic case $B(t)=t$. Indeed we recall the 
\begin{theorem}[Deterministic averaging] Take $B(t)=t$ and $\lb \geq 0$. 
\\
(i) \, For $g=0$ 
$$
 \| e^{-\lb t} \rho_\psi\|^2_{L^2\big(R^+;  \dot H^{1/2}(\R^d) \big)} \leq C({\rm supp} \; \psi ) \; \|ÊÊf^0 \|^2_{L^2(\R^d\times \R^d)} .
$$
(ii) For $f^0=0$
$$
  \| \rho_\psi\|^2_{L^2\big(R^+;  \dot H^{1/2}(\R^d) \big)} \leq  C({\rm supp} \; \psi ) \; \|Ê\psi g \|^{1/2}_{L^2(\R^+ \times \R^d\times \R^d)} \|\psi Êf \|^{3/2}_{L^2(\R^+ \times \R^d\times \R^d)} .
$$
(iii) For $f^0=0$ and $g = \dv h$, we have 
$$
 \| \rho_\psi\|^2_{L^2\big(R^+;  \dot H^{1/4}(\R^d) \big)} \leq C(\psi) \;  \left[ \|Ê h\|^2_{\dot H_x^{-1/2}\cap L^2(\R^+ \times \R^d\times \R^d)} + \| \psi f \|^2_{L^2(\R^+ \times \R^d\times \R^d)}  \right] .
$$
\label{th:averd} 
\end{theorem}

The first two classes of formulas are not the most traditional and express in particular time decay, which is usually not the aim of averaging lemmas which aim at compactness. This is because other methods based on dispersion are more adapted to prove time decay as those in \cite{BaD,Gl,BP_bams}.

One can find related $L^p$ formulas; to do so it is enough to  interpolate them with the other two  elementary pathwise inequalities in $L^1$ and $L^\infty$
$$
\int_{\R^d} | \rho_\psi (x,t) |dx \leq \int_{\R^{2d}} | f^0(x,\xi)| dx d\xi,  \qquad \|  \rho_\psi (t) \|_{L^\infty(\R^d)} \leq \| \psi \|_{L^1(\R^{d})}\;  \| f^0 \|_{L^\infty(\R^{2d})}. 
$$

The usual proof of averaging lemmas, as initiated in  \cite{GLPS}, uses the Fourier transform in space and time. It is not very convenient because the transport term depends on time here. Therefore we prefer a method without time Fourier transform. Such a method has been developed in \cite{bd} but it uses the Fourier transform in $\xi$. We prefer to still develop another variant that is particularly fit to the stochastic case. 
\\

\begin{proof}
We use the Fourier transform  $\wh{f}(k,\xi,t)$ of $f$ in the variable $x$. The equation on the Fourier transform of $f$ becomes
\begin{equation}\label{skineticf}
\f{\p}{\p t}\wh{f}(k,\xi,t) + i\dot B(t)  k. \xi   \wh{f} = \wh{g}.
\end{equation}

As usual we add and subtract a damping term with $\lb >0$ (except for the first result where we do not need the extra-term in the right hand side because the damping term gives it naturally)
$$
\f{\p}{\p t}\wh{f}(k,\xi,t) +i \dot B(t)  \circ k.\xi   \wh{f} + \lb  \wh{f} = \wh{g} +  \lb  \wh{f}.
$$
Because we use the Stratonovich rule, the solution is given by the representation formula
$$
\wh{f}(k,\xi,t) =\wh{f^0}(k,\xi) e^{-\lb t-iB(t) k.\xi}
 + \int_0^t e^{-\lb s}  \big[\wh{g} +  \lb  \wh{f}] (k, \xi, t-s)   \big]e^{i k.\xi  \big(B(t-s) -B(t) \big)} ds 
$$
and thus
\beq 
\wh{\rho_\psi}(k,t) =  \int \psi \wh{f^0}(k,\xi) e^{-\lb t-iB(t) k.\xi} d \xi + \int_0^t \int e^{-\lb s}  \big[\psi \wh{g} +  \lb \psi \wh{f}] e^{i k.\xi  \big(B(t-s) -B(t) \big)}   d\xi ds.
\label{general_rho}
\eeq
We obtain the general inequality that we will however particularize later
$$
\f 12 |\wh{\rho_\psi}(k,t) |^2 \leq  \left|  \int \psi \wh{f^0}(k,\xi) e^{-\lb t-iB(t) k.\xi} d \xi \right|^2 +  \left| \int_0^t \int e^{-\lb s}  \big[\psi \wh{g} +  \lb \psi \wh{f}](k,\xi,t-s) e^{i k.\xi  \big(B(t-s) -B(t) \big)} ds  d\xi \right|^2 .
$$

\noindent {\bf First class of results; $g=0$.} For the first result, we use formula  \eqref{general_rho} with $\lb=0$ and write 
\begin{align*}
 \E \int_{t=0}^\infty & |e^{-\lb t}\wh{\rho_\psi}(k,t) |^2 =  \E \int_{t=0}^\infty \left|  \int \psi \wh{f^0}(k,\xi) e^{-\lb t-iB(t) k.\xi} d \xi \right|^2 dt 
\\[2mm]
&=  \E \int_{t=0}^\infty \left[\int \psi \wh{f^0}(k,\xi_1) e^{-\lb t-iB(t) k.\xi_1} d \xi_1 \;\overline{ \int \psi \wh{f^0}(k,\xi_2) e^{-\lb t-iB(t) k.\xi_2} d \xi_2 } \; \right] dt
\\[2mm]
&= \E   \int_{t=0}^\infty \int \int \psi \wh{f^0}(k,\xi_1) \;\overline{ \psi \wh{f^0}(k,\xi_2) }e^{-2\lb t-iB(t) k.(\xi_1-\xi_2)} d \xi_2 d \xi_1 dt
\\[2mm]
&=  \int_{t=0}^\infty \int \int \psi \wh{f^0}(k,\xi_1) \; \overline{ \psi \wh{f^0}(k,\xi_2) } \int_\R e^{-2\lb t-i w k.(\xi_1-\xi_2)} e^{-\f{w^2}{2t}} \f{dw}{\sqrt{ 2\pi t}} d \xi_2 d \xi_1 dt
\\[2mm]
&=\int_{t=0}^\infty \int \int \psi \wh{f^0}(k,\xi_1) \; \overline{\psi \wh{f^0}(k,\xi_2) } e^{-2\lb t} \exp\big(-\f{t \big|k.(\xi_1-\xi_2)\big|^2}{2} \big)  d \xi_2 d \xi_1 dt
\\[2mm]
&= 2 \int  \int \psi \wh{f^0}(k,\xi_1) \;\overline{ \psi \wh{f^0}(k,\xi_2) } \; \f{1}{4\lb+ |k.(\xi_1-\xi_2)|^2}  d \xi_2 d \xi_1
\\[2mm]
&\leq \f{ C}{\sqrt{\lb}Ê|k|} \int \int |\psi \wh{f^0}(k,\xi) |^2 d \xi
\end{align*}
The last line is not a direct Cauchy-Schwarz inequality. It uses the usual observation leading to averaging lemmas. It can be proved in choosing an orthonormal basis for the $\xi$ space which first vector is $k/|k|$. In the orthogonal hyperplane the Cauchy-Schwarz inequality is enough and thus the constant $C({\rm supp} \; \psi  )$. In the first direction one uses the Young inequality $\|u*v\|_2 \leq C \| u\|_2 \|v\|_1$ and reduces the inequality to compute (with $k'= k/\lb$)
$$
\int_R^R \f{1}{4\lb+ |k|^2  |\eta|^2}  d \eta =\f 1 \lb  \int_R^R \f{1}{4+ |k'|^2  |\eta|^2}  d \eta \leq  \f{ C}{\sqrt{\lb}Ê(\sqrt{\lb}Ê + |k| )} \leq  \f{ C}{\sqrt{\lb}Ê|k|}.
$$
The first result for $g=0$ is proved.
\\

For the second result with $g=0$, we use  \eqref{general_rho}  and estimate the lefthand side term $\lb \wh f$ as in the proof of the  second class of results below. We borrow from there the estimate for the second term which gives
$$
 \E \int_{t=0}^\infty  Ê|k|\;  | \wh{\rho_\psi}(k,t) |^2 \leq  \f{ C}{\sqrt{\lb}} \int \int |\psi \wh{f^0}(k,\xi) |^2 d \xi + \f{ C\lb }{ \sqrt{\lb}} \int |\psi  \wh{f}(k,\xi, t) |^2 dk \; d \xi \; dt .
$$
We choose $\lb=\|g\|_2/\|f\|_2$ and obtain the announced result.
\\
\\
\noindent {\bf Second  class of result; $f^0=0$.} The integral term for $g$ is more technical due to an additional time integration, but the calculation follows the same ideas. For the result with $\lb$, we consider only the  term containing $g$ in \eqref{general_rho} with $\lb=0$,
\begin{align*}
\E \int_{t=0}^\infty & |e^{-\lb t}\wh{\rho_\psi}(k,t) |^2 
= \E \int_{t=0}^\infty  \Big| \int_0^t \int e^{-\lb t} \psi \wh{g}(k,\xi,t-s)  e^{i k.\xi  \big(B(t-s) -B(t) \big)}  d\xi  ds \Big|^2 dt 
\\[2mm]
&= \E   \int_{t=0}^\infty \int  \int_{s_1,s_2=0}^t    e^{-2\lb t} \psi \wh{g}(k,\xi_1,t-s_1)  e^{i k.\xi_1  \big(B(t-s_1) -B(t) \big)} 
\\[2mm]
&\qquad \qquad \qquad   \overline{ \psi \wh{g}(k,\xi_2,t-s_2)  e^{i k.\xi_2  \big(B(t-s_2) -B(t) \big)} } ds_1ds_2 d \xi_2 d \xi_1 dt.
\end{align*}
To continue we may restrict ourselves to $0<s_1<s_2<t$ and compute
\begin{align*}
&\quad \E e^{i k.\xi_1  \big(B(t-s_1) -B(t) \big)}  e^{-i k.\xi_2  \big(B(t-s_2) -B(t) \big)}   
\\
&= \E e^{i k.\xi_1  \big(B(t-s_1) -B(t) \big)}  e^{-i k.\xi_2  \big(B(t-s_2) -B(t-s_1)\big)}   e^{-i k.\xi_2  \big(B(t-s_1) -B(t) \big)}   
\\[1mm]
&= \int e^{i k.(\xi_2-\xi_1) w} e^{-\f{w^2}{2s_1}} \f{dw}{\sqrt{ 2\pi s_1}} \;  \int e^{i k.\xi_2 w} e^{-\f{w^2}{2(s_2-s_1)}} \f{dw}{\sqrt{ 2\pi (s_2-s_1)}} 
\\[1mm]
& = \exp\left({-\f{s_1 |k.(\xi_1-\xi_2)|^2}{2}}\right)  \; \exp\left({-\f{(s_2-s_1) |k.\xi_2|^2}{2}}  \right) . 
\end{align*}
At this stage it is easier to change variables and set $\tau_i =t-s_i$, $t>\tau_1>\tau_2>0$. Then,  combining the two ingredients above, we obtain the expression for the term of interest
\begin{align*}
&\qquad \qquad \f 12  \E \int_{t=0}^\infty |e^{-\lb t}\wh{\rho_\psi}(k,t) |^2  =
\\[2mm] 
&=  \int_{t=0}^\infty \int  \int_{\tau_1>\tau_2=0}^t e^{- 2\lb t} \psi \wh{g}(k,\xi_1,\tau_1)  \overline{ \psi \wh{g}(k,\xi_2,\tau_2) } e^{-\f{(t-\tau_1) |k.(\xi_1-\xi_2)|^2}{2}}  \; e^{-\f{(\tau_1-\tau_2) |k.\xi_2|^2}{2}}  
 d\tau_1d\tau_2 d \xi_2 d \xi_1 dt
\\[2mm]
&=  \int_{\tau_1>\tau_2=0}^\infty \int  \int_{t=\tau_1}^\infty e^{- 2\lb t} \psi \wh{g}(k,\xi_1,\tau_1)  \overline{ \psi \wh{g}(k,\xi_2,\tau_2) }
e^{-\f{(t-\tau_1) |k.(\xi_1-\xi_2)|^2}{2}}  \; e^{-\f{(\tau_1-\tau_2) |k.\xi_2|^2}{2}}   d\tau_1d\tau_2 d \xi_2 d \xi_1 dt
\\[2mm]
&= \int_{\tau_1>\tau_2=0}^\infty \int e^{- 2\lb \tau_1} \psi \wh{g}(k,\xi_1,\tau_1)|Ê\; \overline{\psi \wh{g}(k,\xi_2,\tau_2)} \f{2}{2 \lb+|k.(\xi_1-\xi_2)|^2 }
\; e^{-\f{(\tau_1-\tau_2) |k.\xi_2|^2}{2}}  d\tau_1d\tau_2 d \xi_2 d \xi_1 
\\[2mm]
&= \int_{\tau_1>\tau_2=0}^\infty \int e^{- \lb \tau_1} e^{- \lb \tau_2}  \psi \wh{g}(k,\xi_1,\tau_1)|Ê\; \overline{\psi \wh{g}(k,\xi_2,\tau_2)} \f{2}{2 \lb+|k.(\xi_1-\xi_2)|^2 }
\; e^{-\f{(\tau_1-\tau_2)(2\lb+ |k.\xi_2|^2)}{2}}  d\tau_1d\tau_2 d \xi_2 d \xi_1 .
\end{align*}
Next, we treat the time convolution using the Young inequality $\int u_1(\tau_1) \; u_2*K(\tau_1)  d\tau_1 \leq \|u_1\|_2 \|u_2\|_2 \| K\|_1$ with the convolution $K$ given by the truncated exponential 
\begin{align*}
&\qquad \qquad  \f 12 \E \int_{t=0}^\infty |e^{-\lb t}\wh{\rho_\psi}(k,t) |^2
\\[2mm]
& \leq \int \left( \int |e^{- \lb \tau_1} \psi \wh{g}(k,\xi_1,\tau_1)|^2 d\tau_1\; \int |e^{- \lb \tau_2}\psi \wh{g}(k,\xi_2,\tau_2)|^2 d\tau_2\right)^{1/2}Ê
\f{2}{2 \lb+|k.(\xi_1-\xi_2)|^2 }\;  \f{2}{\lb+|k.\xi_2|^2 } d \xi_2 d \xi_1 
\\[2mm]
& \leq \f 1 \lb  \int \left( \int |e^{- \lb \tau_1}\psi \wh{g}(k,\xi_1,\tau_1)|^2 d\tau_1 \right)^{1/2}Ê\left( \int |e^{- \lb \tau_2} \psi \wh{g}(k,\xi_2,\tau_2)|^2 d\tau_2\right)^{1/2}Ê
\f{4}{2 \lb+|k.(\xi_1-\xi_2)|^2 }\; d \xi_2 d \xi_1 
\\[2mm]
& \leq \f {C({\rm supp}\;  \psi) }{ \lb^{3/2}Ê|k|} \int |e^{- \lb \tau}\psi \wh{g}(k,\xi,\tau)|^2 d\tau d\xi. 
\end{align*}
As before, the last inequality follows from the Young inequality used in $\eta$, the one dimensional component of $\xi_1-\xi_2$ parallel to $k$,  with
$$
\int \f{d\eta}{2 \lb+|k|^2 \eta^2 } =  \f{ C}{\sqrt{\lb}Ê|k|} .
$$
This proves the first inequality with $f^0 \equiv 0$. 
\\

For the second inequality, we use again \eqref{general_rho} and begin with the term containing $g$ in the right hand side.
\begin{align*}
 {\mathcal I}:= \E \int_{t=0}^\infty  \Big| \int_0^t \int &e^{-\lb s} \psi \wh{g}(k,\xi,t-s)  e^{i k.\xi  \big(B(t-s) -B(t) \big)}  d\xi  ds \Big|^2 dt =
\\[2mm]
&= \E   \int_{t=0}^\infty \int  \int_{s_1,s_2=0}^t    e^{-\lb (s_1+s_2)} \psi \wh{g}(k,\xi_1,t-s_1)  e^{i k.\xi_1  \big(B(t-s_1) -B(t) \big)} 
\\[2mm]
&\qquad \qquad \qquad   \overline{ \psi \wh{g}(k,\xi_2,t-s_2)  e^{i k.\xi_2  \big(B(t-s_2) -B(t) \big)} } ds_1ds_2 d \xi_2 d \xi_1 dt.
\end{align*}

To continue we may again restrict ourselves to $0<s_1<s_2<t$ and use the calculation above to compute the expectations. With the variables  $\tau_i =t-s_i$, $t>\tau_1>\tau_2>0$ we obtain
\begin{align*}
&\qquad \qquad \qquad  {\mathcal I}/{2} =
\\[2mm] 
&=  \int_{t=0}^\infty \int  \int_{\tau_1>\tau_2=0}^t e^{-\lb (2t-\tau_1- \tau_2)} \psi \wh{g}(k,\xi_1,\tau_1)  \overline{ \psi \wh{g}(k,\xi_2,\tau_2) } e^{-\f{(t-\tau_1) |k.(\xi_1-\xi_2)|^2}{2}}  \; e^{-\f{(\tau_1-\tau_2) |k.\xi_2|^2}{2}}  
 d\tau_1d\tau_2 d \xi_2 d \xi_1 dt
\\[2mm]
&=  \int_{\tau_1>\tau_2=0}^\infty \int  \int_{t=\tau_1}^\infty e^{-\lb (\tau_1- \tau_2)} \psi \wh{g}(k,\xi_1,\tau_1)  \overline{ \psi \wh{g}(k,\xi_2,\tau_2) }
e^{-\f{(t-\tau_1)[4\lb+ |k.(\xi_1-\xi_2)|^2]}{2}}  \; e^{-\f{(\tau_1-\tau_2) |k.\xi_2|^2}{2}}   d\tau_1d\tau_2 d \xi_2 d \xi_1 dt
\\[2mm]
&= \int_{\tau_1>\tau_2=0}^\infty \int \psi \wh{g}(k,\xi_1,\tau_1)Ê\; \overline{\psi \wh{g}(k,\xi_2,\tau_2)}  \f{2}{4 \lb+|k.(\xi_1-\xi_2)|^2 }
\; e^{-\f{(\tau_1-\tau_2)(\lb+ |k.\xi_2|^2)}{2}}  d\tau_1d\tau_2 d \xi_2 d \xi_1 
\\[2mm]
& \leq \int \left( \int |\psi \wh{g}(k,\xi_1,\tau_1)|^2 d\tau_1\right)^{1/2}Ê\left( \int |\psi \wh{g}(k,\xi_2,\tau_2)|^2 d\tau_2\right)^{1/2}Ê
\f{2}{4 \lb+|k.(\xi_1-\xi_2)|^2 }\;  \f{2}{2 \lb+|k.\xi_2|^2 } d \xi_2 d \xi_1 .
\end{align*}
Again this follows from the Young inequality $\int u_1(\tau_1) \; u_2*K(\tau_1)  d\tau_1 \leq \|u_1\|_2 \|u_2\|_2 \| K\|_1$ with the convolution $K$ given by the truncated gaussian. Then, we conclude because 
\begin{align*}
 {\mathcal I}/{2}  
& \leq \f{2}{\lb} \int \left( \int |\psi \wh{g}(k,\xi_1,\tau_1)|^2 d\tau_1\right)^{1/2}Ê\left( \int |\psi \wh{g}(k,\xi_2,\tau_2)|^2 d\tau_2\right)^{1/2}Ê
\f{1}{4 \lb+|k.(\xi_1-\xi_2)|^2 }\;   d \xi_2 d \xi_1 
\\[2mm]
&\leq  \f{ C}{\lb \sqrt{\lb}Ê|k|} \int |\psi \wh{g}(k,\xi, t) |^2 d \xi \; dt .
\end{align*}
As before, the last inequality follows from the Young inequality used in the one dimensional component of $\xi_1-\xi_2$ parallel to $k$.

These are the usual estimates for the averaging lemma but we win an extra-factor $ \f{1}{\sqrt \lb}Ê$ compared to the deterministic case. 
It does not play a role for regularity in the second result of Theorem \ref{th:aver} because there we choose $\lb=\|g\|_2/\|f\|_2$ in the upper bound
\beq
\E \int \int_0^\infty |k| |\wh{\rho_\psi}(k,t) |^2 dt dk \leq \f{ C}{\lb \sqrt{\lb}}\int | \wh{g}(k,\xi, t) |^2 dk \;d \xi \; dt + \f{ C\lb }{ \sqrt{\lb}} \int | \wh{f}(k,\xi, t) |^2 dk \; d \xi \; dt 
\label{estimate_2}
\eeq
that is obtained after combining the two terms in $g$ and $\lb f$ in \eqref{general_rho}.
\\
\\
\noindent {\bf Third  result; $g={\rm div}\;  h$.} We go back to the formula for $\wh{\rho_\psi}(k,t) $ and write $\psi \dv_\xi \wh{ h}  =-\wh{ h}. \nabla_\xi \psi+\dv[\psi \wh{ h} ]$.  The first term in this Leibniz formula can be treated as before and we examine the second one. It gives a contribution  
\begin{align*}
  \wh{\rho_\psi}(k,t) &= \text{first term}+  \int_0^t \int e^{-\lb s}  \f{\p}{\p \xi}[\psi\wh{ h} ] e^{i k.\xi  \big(B(t-s) -B(t) \big)} ds  d\xi 
\\[2mm]
&= \text{first term} - i  \int_0^t \int e^{-\lb s} \psi\wh{ h}.  k \big(B(t-s) -B(t) \big) e^{i k.\xi  \big(B(t-s) -B(t) \big)} ds  d\xi 
\end{align*}
 Therefore, its contribution to $\E \int_0^\infty |\wh{\rho_\psi}(k,t) |^2 dt$ is 
\begin{align*}
&  \E   \int_{t=0}^\infty \int  \int_{s_1,s_2=0}^t    e^{-\lb (s_1+s_2)} \psi \wh{h}(k,\xi_1,t-s_1).k \big(B(t-s_1) -B(t) \big)  e^{i k.\xi_1  \big(B(t-s_1) -B(t) \big)}  
\\
& \hspace{4cm} { \overline{\psi \wh{h}(k,\xi_2,t-s_2)}. k \big(B(t-s_2) -B(t) \big)  e^{- i k.\xi_2  \big(B(t-s_2) -B(t) \big)} } d\xi_1 \, d\xi_2\, ds_1\, ds_2\, dt 
\end{align*} 
which we evaluate as before. The contribution for $0<s_1<s_2<t$  is the $s_i$, $\xi_i$ integral of the expectation that we evaluate as  
\begin{align*}
& \int_w \int_z  w e^{-i k.(\xi_1-\xi_2)w} (z + w)e^{ i k.\xi_2 z}   e^{-\f{w^2}{2s_1}}e^{-\f{z^2}{2(s_2-s_1)}} \f{dw}{\sqrt{ 2\pi s_1}}  \f{dz}{\sqrt{ 2\pi (s_2-s_1)}}
\\[2mm]
&= \int_w \int_z  w^2 e^{-i k.(\xi_1-\xi_2)w} e^{ i k.\xi_2 z}  e^{-\f{w^2}{2s_1}}e^{-\f{z^2}{2(s_2-s_1)}}  \f{dw}{\sqrt{ 2\pi s_1}} \f{dz}{\sqrt{ 2\pi (s_2-s_1)}}
\\[2mm]
&= s_1 (1- s_1 |k.(\xi_1-\xi_2)|^2)  \exp\left(-\f{s_1 |k.(\xi_1-\xi_2)|^2}{2}\right)  \; \exp\left(-\f{(s_2-s_1) |k.\xi_2|^2}{2}\right)  .
\end{align*}  

After changing the variable  $s_i$ in $\tau_i = t-s_i$, we have to evaluate
\begin{align*}
& \int_{t=0}^\infty \int  \int_{\tau_2\leq \tau_1\leq t }    e^{-\lb (2t-\tau_1- \tau_2)} \psi \wh{h}(k,\xi_1,\tau_1).k \; 
  \overline{\psi \wh{h}(k,\xi_2,\tau_2)}. k 
\\[2mm] 
& (t-\tau_1) \big[1- (t-\tau_1) |k.(\xi_1-\xi_2)|^2 \big]  \exp\left(-\f{(t-\tau_1) |k.(\xi_1-\xi_2)|^2}{2}\right)  \; \exp\left(-\f{(\tau_1-\tau_2) |k.\xi_2|^2}{2}\right)
\end{align*} 
But we notice that
\begin{align*}
&\left|\int_{t=\tau_1}^\infty (t- \tau_1) \big[1- (t- \tau_1) |k.(\xi_1-\xi_2)|^2  \big]  \exp\left( -\f{(t- \tau_1) [4\lb+|k.(\xi_1-\xi_2)|^2]}{2}\right)  dt \right|
\\[2mm]
& \leq \f{C}{\big(4\lb+|k.(\xi_1-\xi_2)|^2\big)^2}.
\end{align*}  

Finally the formula for  $\E \int_0^t |\wh{\rho_\psi}(k,t) |^2 dt$ is controlled by 
\begin{align*}
& \text{first term}+ \int \int_{0< \tau_2< \tau_1} \psi \wh{h}(k,\xi_1,\tau_1).k \; \overline{\psi \wh{h}(k,\xi_2,\tau_2)}. k \; \f{C}{\big(4\lb+|k.(\xi_1-\xi_2)|^2\big)^2}e^{-\lb(\tau_1- \tau_2)}
\\ \\
& \leq \text{first term}+ C\f{|k|}{\lb^2 \sqrt \lb} \| \psi \wh{h}(k,\xi,t-s)\|^2_{L^2_{t,\xi}}
\end{align*}  
because 
$$
\int  \f{d\eta}{\big(4\lb+|k|^2 \eta^2\big)^2} =\f{C}{ |k| \lb \sqrt \lb}, \qquad  \int_0^\infty  e^{-{\tau}{\lb}} d \tau= \f{1}{\lb}.
$$

When balancing this contribution with that of $\lb \wh{f}$ we end up with 
$$
C\left[ \f{|k|}{\lb^2 \sqrt \lb} + \f{\sqrt \lb}{|k|}  \right] \left[\| \psi \wh{h}(k,\xi,t)\|^2_{L^2_{t,\xi}}  +\| \psi \wh{f}(k,\xi,t)\|^2_{L^2_{t,\xi}}  \right]
$$
which is optimal for $\lb= |k|^{2/3}$ and gives
$$
C \f{1}{|k|^{2/3}}  \left[\| \psi \wh{h}(k,\xi,t)\|^2_{L^2_{t,\xi}}  +\| \psi \wh{f}(k,\xi,t)\|^2_{L^2_{t,\xi}}  \right].
$$
This means a regularizing term in $\dot H^{1/3} $ regularity  as announced.
\\

The contribution of the First Term, analogous to  the second class of results \eqref{estimate_2},  gives a contribution controlled by
$$
\f{C}{\lb {\sqrt \lb } \; |k| }  \; \| \nabla \psi \, \wh{h}(k,\xi,t)\|^2_{L^2_{t,\xi}} 
$$
that gives with our previous choice of $\lb = |k|^{2/3} $
$$
\f{C}{|k|^{2/3}}\;\f{1} {|k|^{4/3} } \; \| \nabla \psi  \, \wh{h}(k,\xi,t)\|^2_{L^2_{t,\xi}} .$$

These controls together give the third result.
\end{proof}

\section{Space regularity, general velocity field $a(\xi)$}
\label{sec:sal}

Our  method can be extended to the case of the kinetic equation with a more general transport field, still motivated by the case of scalar conservation laws, 
\begin{equation}\label{skineticg}
\begin{cases}
\f{\p}{\p t} f(x,\xi,t)  + \dot B(t) \circ a(\xi) . \nabla_x f = g(x, \xi,t )  \quad \text{ in } \quad \R^d \times  \R \times(0,\infty), 
\\[2mm]
f(t=0)=f^0 \quad \text{ on } \quad  \R^d \times  \R .
\end{cases}
\end{equation}
We assume that the transport field  $a(\xi) \in C^1(\R; \R^d)$ 
satisfies a natural non-degeneracy condition (in view of the deterministic averaging lemmas): there is a $\al \geq 1$ and a constant $A(\al)>0$ such that
\begin{equation}\label{as_kinetic}
\inf_{e\in \R^d, \; |e|=1} \; |e.a(\xi_1)-e.a(\xi_2)|\geq A \;  |\xi_1 - \xi_2|^\alpha, \qquad 1\leq \alpha < \infty.
\end{equation}

Because it uses pointwise control, this assumption is stronger than the usual one based on measures of  the sets $\{Ê|e.a(\xi)- \tau| \leq \e \}$ but it is in the same spirit of imposing nonlinearity of the curves $\xi \mapsto a(\xi)$ (that is this $d$-dimensional curve is not locally contained in a hyperplane) as in the deterministic case  \cite {bd,DLM,Ger,GPS,GLPS,BP_bams,PS}.

\begin{theorem}[Stochastic averaging for general velocity field] Take $B$ a brownian motion and assume \eqref{as_kinetic} in the stochastic kinetic equation \eqref{skineticg}. Take  $\lb >0$ 
\\
(i) \, For $g=0$, we have 
$$
\E \| e^{-\lb t} \rho_\psi\|^2_{L^2\big(R^+; \dot H^{1/(2\al)}(\R^d) \big)} \leq \f{C}{\lb^{1-1/2\al}} \|ÊÊ\psi  f^0 \|^2_{L^2(\R^d\times \R)} ,
$$
$$
\E \| \rho_\psi\|^2_{L^2\big(R^+; \dot H^{1/(2\al)}(\R^d) \big)} \leq C \|  \psi {f^0}\|_{L^2(\R^d\times \R)}^{1+1/2\al}  \|  \psi f \|_{L^2(\R^d\times \R\times \R^+)}^{1-1/2\al}.
$$
(ii) For $f^0=0$, we have 
$$
\E \| e^{-\lb t} \rho_\psi\|^2_{L^2\big(R^+; \dot H^{1/(2\al)}(\R^d) \big)} \leq \f{C}{\lb^{1-1/2\al}} \|Êe^{-\lb t}Ê\psi  g \|^2_{L^2(\R^d\times \R \times \R^+)} ,
$$
$$
\E \| \rho_\psi\|^2_{L^2\big(R^+; \dot H^{1/(2\al)}(\R^d) \big)} \leq C \|  \psi {g}\|_{L^2(\R^d\times \R \times \R^+)}^{1+1/2\al}  \|  \psi  f \|_{L^2(\R^d\times \R\times \R^+)}^{1-1/2\al}.
$$
\label{th:averg} 
\end{theorem}

\begin{proof} {\bf First class of results.} 
The representation formula  \eqref{general_rho} is simply changed to 
$$
e^{-\lb t} \wh{\rho_\psi}(k,t) =  \int \psi \wh{f^0}(k,\xi) e^{-\lb t-iB(t) k.a(\xi)} d \xi + \int_0^t \int e^{-\lb s}  \big[\psi \wh{g} +  \lb \psi \wh{f}] e^{i k.a(\xi)  \big(B(t-s) -B(t) \big)} ds  d\xi ,
$$
therefore, we can carry out the same analysis until we evaluate the convolution in $\xi$. For the term in $f^0$ for instance, we find 
\begin{align*}
\E \int_0^\infty e^{-2\lb t} |\wh{\rho_\psi}(k,t) |^2 &= 2 \int \psi \wh{f^0}(k,\xi_1) \;\overline{ \psi \wh{f^0}(k,\xi_2) } \; \f{1}{4\lb+ |k.(a(\xi_1)- a(\xi_2)|^2}  d \xi_2 d \xi_1
\\ \\
& \leq C \int \big| \psi \wh{f^0}(k,\xi_1)\big| \; \big| \overline{ \psi \wh{f^0}(k,\xi_2) }\big| \; \f{1}{4\lb+ |k|^2Ê\; |\xi_1- \xi_2|^{2\al}}  d \xi_2 d \xi_1
\\ \\
& \leq C \|  \psi \wh{f^0}(k,\xi)\|^2_{L^2_\xi} \;  \big\| \f{1}{4\lb+ |k|^2Ê\; |\xi|^{2\al}} \big\|_{L^1_\xi}.
\end{align*}
We  obtain the estimate 
$$
\big \| \f{1}{4\lb+ |k|^2Ê\; |\xi|^{2\al}} \big \|_{L^1_\xi} = \f{C}{\lb}  \left( \f{\sqrt \lb}{|k| } \right)^{1/\al } ,
$$
and the first result follows.
\\

For the second result, we use the inequality obtained in section \ref{sec:space} on the contribution to $\E \int_0^\infty  |\wh{\rho_\psi}(k,t) |^2$ from the quantity (with $g = \lb f$)
$$
{\mathcal I} :=  \E \int_{t=0}^\infty  \Big| \int_0^t \int e^{-\lb s} \psi \wh{g}(k,\xi,t-s)  e^{i k.\xi  \big(B(t-s) -B(t) \big)}  d\xi  ds \Big|^2 dt ,
$$ 
\begin{align*}
{\mathcal I}& \leq \f{2}{\lb} \int \left( \int |\psi \wh{g}(k,\xi_1,\tau_1)|^2 d\tau_1\right)^{1/2}Ê\left( \int |\psi \wh{g}(k,\xi_2,\tau_2)|^2 d\tau_2\right)^{1/2}Ê
\f{1}{4 \lb+|k.(\xi_1-\xi_2)|^2 }\;   d \xi_2 d \xi_1 
 \\
&\leq \f{2}{\lb} \int \left( \int |\psi \wh{g}(k,\xi_1,\tau_1)|^2 d\tau_1\right)^{1/2}Ê\left( \int |\psi \wh{g}(k,\xi_2,\tau_2)|^2 d\tau_2\right)^{1/2}Ê
\f{1}{4 \lb+|k|^2\; | \xi_1-\xi_2|^{2\al} }\;   d \xi_2 d \xi_1 
\\
& \leq C \|  \psi \wh g(k,\xi,t)\|^2_{L^2_{\xi,t}} \;  \| \f{1}{4\lb+ |k|^2Ê\; |\xi|^{2\al}} \|_{L^1_\xi}
\\
& \leq \f{C}{\lb}  \left( \f{\sqrt \lb}{|k| } \right)^{1/\al } \;  \|  \psi \wh g(k,\xi,t)\|^2_{L^2_{\xi,t}} .
\end{align*}

This gives
$$
|k| ^{1/\al } \;  \E \int_0^\infty |\wh{\rho_\psi}(k,t) |^2 \leq  \f{C}{\lb}  \lb^{1/2\al }  \|  \psi \wh{f^0}(k,\xi)\|^2_{L^2_\xi} + {C}{\lb} \;  \lb^{1/2\al } \;  \|  \psi \wh f(k,\xi,t)\|^2_{L^2_{\xi,t}} .
$$ 
And, after integrating in $k$ and then optimizing the value of $\lb$
$$
\int_{\R^d} |k| ^{1/\al } \;  \E \int_0^\infty |\wh{\rho_\psi}(k,t) |^2 d k \leq C \|  \psi \wh{f^0}\|_{L^2}^{1+1/2\al}  \|  \psi \wh f \|_{L^2}^{1-1/2\al}.
$$
This is the second result.
\\
\\
{\bf Second class of results.}  For the result with $\lb$, we use again our previous inequality in section \ref{sec:space} which we modify to change $\xi$ in $a(\xi)$. We obtain 
\begin{align*}
&\qquad \qquad  \f 12 \E \int_{t=0}^\infty |e^{-\lb t}\wh{\rho_\psi}(k,t) |^2
\\[2mm]
& \leq \f 1 \lb  \int \left( \int |e^{- \lb \tau_1}\psi \wh{g}(k,\xi_1,\tau_1)|^2 d\tau_1\; \int |e^{- \lb \tau_2} \psi \wh{g}(k,\xi_2,\tau_2)|^2 d\tau_2\right)^{1/2}Ê
\f{4}{2 \lb+|k.(a(\xi_1) - a(\xi_2))|^2 }\; d \xi_2 d \xi_1 
\\[2mm]
& \leq \f 1 \lb  \int \left( \int |e^{- \lb \tau_1}\psi \wh{g}(k,\xi_1,\tau_1)|^2 d\tau_1 \right)^{1/2}Ê\left(  \int |e^{- \lb \tau_2} \psi \wh{g}(k,\xi_2,\tau_2)|^2 d\tau_2\right)^{1/2}Ê
\f{4}{2 \lb+|k|^2\; |\xi_1 - \xi_2|^{2\al} }\; d \xi_2 d \xi_1 
\\[2mm]
& \leq \f{C}{\lb}  \left( \f{\sqrt \lb}{|k| } \right)^{1/\al } \; \int |e^{- \lb \tau}\psi \wh{g}(k,\xi,\tau)|^2 d\tau d\xi. 
\end{align*}
This gives the first result. We do not prove the second result which follows from the same combination as in section \ref{sec:space} with these ingredients.
\end{proof}

\section{Time regularity}
\label{sec:time}

We can also obtain time regularity with the method developed before. The gain of regularity is however lower than in the deterministic case (see Appendix \ref{ap:time}) which is a more visible symptom of the difference of scale between brownian motion and classical time. For deterministic kinetic transport equations,  space and time play a similar role  but this no longer true in the stochastic case (see the scale $|k| \sim \sqrt \lb$ below rather than $|k| \sim \lb$ in the deterministic case). 
%
\begin{theorem}[Stochastic time averaging for general velocity field] Take $B$ a brownian motion, $\lb>0$  and  $g=0$ in equation \eqref{skinetic}. Then for high Fourier frequencies in $x$,  $e^{-\lb t} \rho_\psi \in H^{1/4}\big(R^+; L^{2}(\R^d) \big)$, more precisely 
$$
\E \big\| \f{|k| \sqrt{\lb} }{\sqrt{\lb} + |k|} e^{-\lb t}  \wh \rho_\psi (k,t) \big\|^2_{H^{1/4}\big(R^+; L^{2}(\R^d) \big)} \leq C \|ÊÊf^0 \|^2_{L^2(\R^d\times \R^d)} .
$$

For $a(\xi)$ satisfying \eqref{as_kinetic}  in the stochastic kinetic equation \eqref{skineticg}, we have for all $ \beta = 1/4 \al$
$$
\E \big\|  \f{|k|^{1/\al} \lb^{1/2\al}  }{\lb^{1/2\al} + |k|^{1/\al}} e^{-\lb t} \wh \rho_\psi (k,t)\big \|^2_{H^{1/4}\big(R^+; L^{2}(\R^d) \big)} \leq C \|ÊÊf^0 \|^2_{L^2(\R^d\times \R)} .
$$
\label{th:averagetime} 
\end{theorem}
%

Our strategy of proof uses the integral characterization of the fractional Sobolev space (\cite{stein}, p139) and thus we define for $0< \beta < 1$, 
$$
\| u \|^2_{\dot H^{\beta}(\R^+)} = \int_{[0,\infty] \times [0,\infty]} \f{|u(t)-u(s)|^2}{[t-s]^{1+2\beta} }ds \, dt. 
$$
The proof can be used also in the deterministic case and gives the correct regularity $H^{1/2}$. See Appendix \ref{ap:time}.
\\

\begin{proof}  We compute for $t>s$ and with $a_1=a(\xi_1)$, $a_2=a(\xi_2)$,
\begin{align*}
&\E | e^{-\lb t} \wh{\rho_\psi}(k,t) - e^{-\lb s} \wh{\rho_\psi}(k,s) |^2 
\\[3mm]
& = \E \int \psi \wh{f^0}(k,\xi_1)  \overline{ \psi \wh{f^0}(k,\xi_2) }  \left(e^{-\lb t -iB(t)k.a_1}-e^{-\lb s - iB(s)k.a_1} \right)\left(e^{-\lb t + iB(t)k.a_2}-e^{-\lb s + iB(s)k.a_2} \right)  d \xi_2 d \xi_1
\\[3mm]
& =  \E \int \psi \wh{f^0}(k,\xi_1) \,\overline{ \psi \wh{f^0}(k,\xi_2) } \, \left(e^{-\lb (t-s) -i(B(t)-B(s))k.a_1} -1 \right)\left(e^{-\lb (t-s) +i(B(t)-B(s)) k.a_2} - 1 \right) 
\\
& \hspace{12cm}  e^{-2 \lb s + iB(s)k.(a_2-a_1)} d \xi_2 d \xi_1
\\[3mm]& = \f 1{2 \pi} \int \psi \wh{f^0}(k,\xi_1) \;\overline{ \psi \wh{f^0}(k,\xi_2) } \; \left(e^{-\lb (t-s) -iw.k.a_1} -1 \right)\left(e^{-\lb (t-s) i w. k.a_2} - 1 \right)  e^{-2 \lb s+ iz k.(a_2-a_1)}  
\\
& \hspace{10cm}e^{-\f{|w|^2}{2(t-s)}} e^{-\f{|z|^2}{2s}}  \f{dw}{ (t-s)^{d/2}}Ê \; \f{dz} {s^{d/2}}Ê \;   d\xi_2 \; d \xi_1 
\end{align*}

For the fractional Sobolev exponent $0<\beta <1$, we need the 
\begin{lemma}
\begin{align*}
\Big|  \int_{s=0}^\infty \int_{t=s}^\infty  \int_{w, z} & \left(e^{-\lb (t-s) -iw.k.a_1} -1 \right) \left(e^{-\lb (t-s) i w. k.a_2} - 1 \right)  
\\
& \hspace{2cm}e^{-2 \lb s+ iz k.(a_2-a_1)}  
  e^{-\f{|w|^2}{2(t-s)}} e^{-\f{|z|^2}{2s}}  \f{dw}{ (t-s)^{d/2}}Ê\; \f{dz} {s^{d/2}} \;  \f{dt \; ds}{|t-s|^{1+2 \beta}}  \Big|
\\
&\leq   \f{C(\beta) \lb^{2\beta-1} \;(1+ | \wt  k|^{4\beta}) }{ 2+ | \wt k.(a_2-a_1)|^2}  
\end{align*}
with $\wt k= k/\sqrt{\lb}$ and $C(\beta)=O(\f{1}{(1-2\beta)^{2\beta}})$.

\label{lm:technic} \end{lemma}
%
\noindent {\bf Proof of Lemma \ref{lm:technic}.} With the change of variable $t \mapsto u=t-s$, we estimate the expression of interest as
\begin{align*}
& \left|  \int_{s=0}^\infty \int_{u=0}^{\infty}  \int_{w, z}  \left(e^{ -\lb u -iw\, k.a_1} -1 \right)\left(e^{-\lb u + i w \, k.a_2} - 1 \right)  e^{-2 \lb siz \, k.(a_2-a_1)}  e^{-\f{|w|^2}{2u}} e^{-\f{|z|^2}{2s}}  \f{dw}{ u^{d/2}}Ê\; \f{dz} {s^{d/2}} \;  \f{du \; ds}{u^{1+2 \beta}}  \right|
\\[2mm]
&= 2 \pi \left|  \int_{s=0}^\infty \int_{u=0}^{\infty}   \left( e^{-u [2 \lb+ | k.(a_2-a_1)|^2/2]} - e^{-u [ \lb+ | k.a_1|^2/2]} - e^{-u[\lb+ | k.a_2|^2/2]} +1  \right)  e^{-s[2 \lb+ |k.(a_2-a_1)|^2/2]}  \f{du \, ds}{u^{1+2 \beta}}  \right|
\\[2mm]
&=    \f{2 \pi}{ 2 \lb+ |k.(a_2-a_1)|^2}  \left|  \int_{u=0}^{\infty} \big[e^{-u [2 \lb+ | k.(a_2-a_1)|^2/2]} - e^{-u [ \lb+ | k.a_1|^2/2]} - e^{-u[\lb+ | k.a_2|^2/2]} +1 \big]\,   \f{du}{u^{1+2 \beta}} \right| 
\\[2mm]
&  \leq  \f{C}{ 2 \lb+ |k.(a_2-a_1)|^2} \Big[\; \big|2 \lb+ | k.(a_2-a_1)|^2/2\big|^{2\beta} + \big| \lb+ | k.a_1|^2/2]\big|^{2\beta} + \big| \lb+ | k.a_2|^2/2\big|^{2\beta} \; \Big]
\\[2mm]
&  \leq  \f{C(\beta) \lb^{2\beta-1} }{ 2+ | \wt k.(a_2-a_1)|^2} \Big[\; |4 + | \wt  k.(a_2-a_1)|^2|^{2\beta} + | 2+ | \wt k.a_1|^2]|^{2\beta} + | 2 + | \wt k.a_2|^2|^{2\beta} \; \Big]
\end{align*}
The result follows. 
\qed
\bigskip

We go back to our estimate on $\E |\wh{\rho_\psi}(k,t) -\wh{\rho_\psi}(k,s) |^2$ and we  have to estimate the quantity 
$$
I(k)= \E \int_{s=0}^\infty  \int_{t=s}^\infty \frac{ |\wh{\rho_\psi}(k,t) -\wh{\rho_\psi}(k,s) |^2 }{|t-s|^{1+2 \beta} }dt \; ds.
$$ 
With the above lemma we conclude that, in the case $a(\xi)= \xi$,  
\begin{align*}
I(k) &  \leq C  \int | \psi \wh{f^0}(k,\xi_1) |\; |\overline{ \psi \wh{f^0}(k,\xi_2) } |\;  \f{C(\beta) \lb^{2\beta-1} \; (1+| \wt  k|^{4\beta}) }{ 2+ | \wt k.(a_2-a_1)|^2}   d \xi_2 d \xi_1
\\[2mm]
& \leq C (\beta, {\rm supp} \psi)\; \| \psi \wh{f^0} \|^2_{L^2_\xi} \; \Big\|Ê\f{ \lb^{2\beta-1} (1+ | \wt  k|^{4\beta}) }{ 2+ | \wt k. \xi |^2 }\Big\|_{L^1_\xi}
\\[2mm]
& \leq C (\beta, {\rm supp} \psi) \;  \| \psi \wh{f^0} \|^2_{L^2_\xi} \f{ 1+ | \wt  k|^{4\beta} }{|k|} \lb^{2\beta-1/2}
\end{align*}
which leads us to choose $\beta=1/4$.
\\

In the general case we conclude with (still with a appropriate choice of coordinates in $\xi$)
\begin{align*}
I(k) &  \leq  \int | \psi \wh{f^0}(k,\xi_1) |\; |\overline{ \psi \wh{f^0}(k,\xi_2) } |\;  \f{C(\beta) \lb^{2\beta-1} \; (1+| \wt  k|^{4\beta}) }{ 2+ | \wt k|^2|(\xi_2-\xi_1)|^{2\al}}   d \xi_2 d \xi_1
\\[2mm]
& \leq C (\beta, {\rm supp} \psi)\; \| \psi \wh{f^0} \|^2_{L^2_\xi} \; \Big\|Ê\f{ \lb^{2\beta-1} (1+ | \wt  k|^{4\beta}) }{ 2+ | \wt k|^2 | \xi |^{2\al} }\Big\|_{L^1_\xi}
\\[2mm]
& \leq C (\beta, {\rm supp} \psi) \;  \| \psi \wh{f^0} \|^2_{L^2_\xi} \f{ 1+ | \wt  k|^{4\beta} }{|k|^{1/\al}} \lb^{2\beta-1/(2\al)}
\end{align*}
which leads to the choice $\beta = 1/4\al$.
\end{proof}

\appendix

\section{Deterministic averaging lemma,  space regularity}
\label{ap:ald}

In the first statement of Theorem \ref{th:aver}, a coefficient $\sqrt \lb$ appears; it expresses that the time decay in the stochastic case is slow. In the deterministic result of Theorem \ref {th:averd} (i), this parameter $\lb$ is not needed and thus, the averaging lemma in space also expresses faster time decay.  Being a time scale parameter, this slower decay is not surprising when a brownian motion is involved. To  better understand how it appears, we give the corresponding proof of the deterministic averaging lemma.

We consider the regularizing   effect with respect to the initial data (first statement of Theorem \ref{th:averd}, i.e., the solution to 
\begin{equation}\label{dkinetic}
\begin{cases}
\f{\p}{\p t} f(x,\xi,t)  + \xi . \nabla_x f + \lb f = 0  \quad \text{ in } \quad \R^{2d} \times(0,\infty), 
\\[2mm]
f(t=0) =f^0 \quad \text{ on } \quad \R^{2d}.
\end{cases}
\end{equation}
The solution is given by
$$
\wh{f}(k,\xi,t)= \wh{f^0}(k,\xi)e^{(-\lb+i\xi.k)t}, 
\qquad 
\wh{\rho_\psi } (k,t)= \int \psi\wh{f^0}(k,\xi)e^{(-\lb+i\xi.k)t} d\xi, 
$$
\begin{align*}
  \int_0^\infty e^{-2 \lb t}  |\wh{\rho_\psi } (k,t)|^2 dt & =  \int \int_{t=0}^\infty \psi\wh{f^0}(k,\xi_1)  \overline{\psi\wh{f^0}(k,\xi_2)} e^{-\big(2\lb- i(\xi_1-\xi_2).k\big)t} dt d\xi_1 \; d\xi_2
\\[2mm]
&=  \int \psi\wh{f^0}(k,\xi_1) \; \overline{\psi\wh{f^0}(k,\xi_2)} \; \f{1}{2\lb- i(\xi_1-\xi_2).k}\;  d\xi_1 \; d\xi_2
\\[2mm]
& \leq  \f{C}{|k|}\| \psi\wh{f^0}(k,\xi) \|^2_{L^2(\R_\xi)} 
\end{align*}

Before we explain this last line, we point out that the difference with the stochastic case appears in the Fourier multiplier, here $\f{1}{2\lb- i(\xi_1-\xi_2).k}$ and in the stochastic case $\f{1}{2\lb+ |(\xi_1-\xi_2).k|^2}$. The homogeneity in $k$ compared to $\xi$ is the same (hence the same gain of regularity) but the decay in $\lb$ is different: hyperbolic in the deterministic case, parabolic in the stochastic case. 
\\

Following the argument for the proof of Theorem \ref{th:aver}, in the case at hand, after reduction to the direction of $k$ for $\xi$, the convolution kernel $\f{1}{2\lb- i \xi |k| }$ is not in $L^1_\xi$. However, it is still bounded by $\f{C}{|k|}$ as an operator on $L^2_\xi$, e.g., because its Fourier transform is $\f{1}{|k|} H(\eta) e^{-\lb \eta/|k|}$ and it is bounded by $1/|k|$ independently of $\lb$; that this operator is bounded in $L^2_\xi$ relies however on the general structure of a Calder\'on-Zygmund operator (see \cite{stein} for instance) and not on its specific form that allows to compute exactly its Fourier transform.

\section{Deterministic averaging lemma,  time regularity}
\label{ap:time}

For the stochastic averaging lemma, we have presented a proof of time regularity that leads to a gain of $\f14$th derivative. We present here the same proof in the deterministic case so as to show it recovers the optimal result with a gain of $\f 12$ derivatives.

\begin{theorem}[Deterministic time averaging] Take $B(t)=t$, $\lb \geq 0$ and $g=0$ in equation \eqref{skinetic}, then  $e^{-\lb t} \rho_\psi \in \dot H^{1/2}\big(R^+; L^{2}(\R^d) \big)$, more precisely 
$$
\| e^{-\lb t}  \wh \rho_\psi (k,t) \|^2_{H^{1/4}\big(R^+; L^{2}(\R^d) \big)} \leq C \|ÊÊf^0 \|^2_{L^2(\R^d\times \R^d)} .
$$
\label{th:averagetimed} 
\end{theorem}
%

\begin{proof}  We compute for $t>s$ 
\begin{align*}
& | e^{-\lb t} \wh{\rho_\psi}(k,t) - e^{-\lb s} \wh{\rho_\psi}(k,s) |^2 
\\[3mm]
& =  \int \psi \wh{f^0}(k,\xi_1)  \overline{ \psi \wh{f^0}(k,\xi_2) }  \left(e^{-\lb t -i t k.\xi_1}-e^{-\lb s - i s k.\xi_1} \right)\left(e^{-\lb t + i t k.\xi_2}-e^{-\lb s + i s k.\xi_2} \right)  d \xi_2 d \xi_1
\\[3mm]
& =  \int \psi \wh{f^0}(k,\xi_1) \,\overline{ \psi \wh{f^0}(k,\xi_2) } \, \left(e^{-\lb (t-s) -i(t-s)k.\xi_1} -1 \right)\left(e^{-\lb (t-s) +i(t-s) k.\xi_2} - 1 \right) 
\\
& \hspace{12cm}  e^{-2 \lb s + is k.(\xi_2-\xi_1)} d \xi_2 d \xi_1 .
\end{align*}

And we compute 
\begin{align*}
\int_{s=0}^\infty&  \int_{t=s}^\infty \left(e^{-\lb (t-s) -i(t-s)k.\xi_1} -1 \right)\left(e^{-\lb (t-s) +i(t-s) k.\xi_2} - 1 \right) 
    e^{-2 \lb s + is k.(\xi_2-\xi_1)} \f{ds \, dt}{(t-s)^{1+2\beta} }
    \\[3mm]
& = \int_{s=0}^\infty \int_{u=0}^\infty \left(e^{-\lb u -i u k.\xi_1} -1 \right)\left(e^{-\lb u +i u k.\xi_2} - 1 \right) 
    e^{-2 \lb s + is k.(\xi_2-\xi_1)} ds \, \f{du}{u^{1+2\beta} }
   \\[3mm]
& = \f{1}{2 \lb - i k.(\xi_2-\xi_1)} \int_{u=0}^\infty \left(e^{-2 \lb u -i u k.(\xi_1-\xi_2)} - e^{-\lb u + i u k.\xi_2} - e^{-\lb u - i u k.\xi_1} + 1 \right) \, \f{du}{u^{1+2\beta} }    
\end{align*}
Compared to the stochastic case, there is one more cancellation in the singularity at $u=0$ which makes that the integral converges for $\beta = 1/2$. Indeed we can write 
\begin{align*}
F(a,b)& =  \int_{u=0}^\infty \left( e^{-(a+b) u} -e^{-a u}  - e^{-b u} + 1 \right) \, \f{du}{u^{2\beta} }    
\\[2mm]
& =  \int_{u=0}^R \left( e^{-(a+b) u} -e^{-a u}  - e^{-b u} + 1 \right) \, \f{du}{u^{2\beta} }   + \int_R^\infty....
\end{align*}
\begin{align*}
|F(a,b)|& \leq C (|a|^2+|b|^2) R^{2-2\beta}+   C R^{-2\beta}
\\[2mm]
& \leq C (|a|^2+|b|^2)^\beta. 
\end{align*}

Therefore, we can act the Fourier multiplier as in Appendix \ref{ap:ald} and find  
\begin{align*}
 \int_{s=0}^\infty  \int_{t=s}^\infty & \frac{ |\wh{\rho_\psi}(k,t) -\wh{\rho_\psi}(k,s) |^2 }{|t-s|^{1+2 \beta} }dt \; ds 
\\[2mm]
&  = \int  \psi \wh{f^0}(k,\xi_1) \; \overline{ \psi \wh{f^0}(k,\xi_2) } \;  \f{ F\big(2 \lb+i k.\xi_1, 2 \lb - i k.\xi_2) }{ \lb+ i k.(\xi_2- \xi_1\big)}   d \xi_2 d \xi_1
\\[2mm]
& \leq C ( {\rm supp} \psi)\; \| \psi \wh{f^0} \|^2_{L^2_\xi} \; Ê\f{ (\lb^2 + | k|^{2})^\beta }{ \lb+  | k | }
\\[2mm]
& \leq C ( {\rm supp} \psi)\; \| \psi \wh{f^0} \|^2_{L^2_\xi}
\end{align*}
after choosing $\beta=1/2$, and we arrive to the result.

\end{proof}

%
%
\bibliographystyle{siam}

\bigskip
\bigskip

\noindent\begin{minipage}[t]{8cm}
\noindent ($^{1}$) Universit\'e Paris-Dauphine, \\
 place du Mar\'echal-de-Lattre-de-Tassigny, \\
 75775 Paris cedex 16, France \\
\end{minipage}
\begin{minipage}[t]{9cm}
\noindent ($^{2}$) Universit\'e Pierre et Marie Curie, Paris 06,  \\
      CNRS UMR 7598 Laboratoire J.-L. Lions, BC187,\\
       4, place Jussieu,  F-75252 Paris 5 \\
	and INRIA Paris-Rocquencourt, EPI Bang
\end{minipage}
\\ \\ \\
\centerline{\begin{minipage}[t]{9cm}
($^{3}$)  Department of Mathematics \\
             University of Chicago \\
             Chicago, IL 60637, USA \\
($^{4}$)  Partially supported by the National Science Foundation.            
 \end{minipage} }
\\ \\

\end{document}